\documentclass{amsart}

\usepackage{fancyhdr}
\usepackage{lastpage}
\usepackage{yhmath}

\newcommand{\bdis}{\begin{displaymath}}
\newcommand{\edis}{\end{displaymath}}
\newcommand{\be}{\begin{equation}}
\newcommand{\ee}{\end{equation}}

\newcommand{\mbb}{\mathbb}
\newcommand{\mcal}{\mathcal}

\newcommand{\vp}{\varphi}
\newcommand{\vt}{\vartheta}

\newcommand{\mT}{\mathring{T}}

\newcommand{\hh}{h_\nu(\tau)}
\newcommand{\hhp}{h_{2\nu}(\tau)}

\newcommand{\zf}{\zeta\left(\frac{1}{2}+it\right)}

\newtheorem{lemma}[]{Lemma}

\theoremstyle{definition}

\theoremstyle{remark}
\newtheorem{remark}[]{Remark}

\newtheorem*{mydef1}{{\bf Theorem}}

\numberwithin{equation}{section}

\begin{document}

\title{Riemann hypothesis and the arc length of the Riemann $Z(t)$-curve}

\author{Jan Moser}

\address{Department of Mathematical Analysis and Numerical Mathematics, Comenius University, Mlynska Dolina M105, 842 48 Bratislava, SLOVAKIA}

\email{jan.mozer@fmph.uniba.sk}

\keywords{Riemann zeta-function}

\begin{abstract}
On Riemann hypothesis it is proved in this paper that the arc length of the Riemann $Z$-curve
is asymptotically equal to the double sum of local maxima of the function $Z(t)$ on corresponding segment.
This paper is English remake of our paper \cite{9}, with short appendix concerning new integral generated by
Jacob's ladders added.
\end{abstract}

\maketitle

\section{Introduction and result}

\subsection{}

Main object of this paper is the study of the integral
\be \label{1.1}
\int_T^{T+H} \sqrt{1+\{ Z'(t)\}^2}{\rm d}t,
\ee
i.e. the study of the arc length of the Riemann curve
\bdis
y=Z(t),\ t\in [T,T+H],\quad T\to\infty,
\edis
where (see \cite{13}, pp. 79, 329)
\be \label{1.2}
\begin{split}
 & Z(t)=e^{i\vt(t)}\zf, \\
 & \vt(t)=-\frac t2\ln\pi+\text{Im}\ln\Gamma\left(\frac 14+i\frac t2\right)= \\
 & = \frac{t}{2\pi}\ln\frac{t}{2\pi}-\frac t2-\frac{\pi}{8}+\mcal{O}\left(\frac 1t\right).
\end{split}
\ee

\begin{remark}
Let us remind that the formula
\be \label{1.3}
\begin{split}
& \{ Z(t)=\}\ e^{i\vt(t)}\zf= \\
& = 2\sum_{n\leq\sqrt{\bar{t}}}\frac{1}{\sqrt{n}}\cos\{\vt(t)-t\ln n\}+\mcal{O}(t^{-1/4}),\ \bar{t}=\sqrt{\frac{t}{2\pi}}
\end{split}
\ee
was known to Riemann (see \cite{11}, p. 60, comp. \cite{12}, p. 98).
\end{remark}

Next, we will denote the roots of the equations
\bdis
Z(t)=0,\ Z'(t)=0,\ t_0\not=\gamma
\edis
by the symbols
\bdis
\{\gamma\},\ \{ t_0\},
\edis
correspondingly.

\begin{remark}
On the Riemann hypothesis, the points of the sequences $\{\gamma\}$ and $\{ t_0\}$ are separated each from other (see \cite{3}, Corollary 3), i.e.
in this case we have
\bdis
\gamma'<t_0<\gamma'' ,
\edis
where $\gamma',\gamma''$ are neighboring points of the sequence $\{\gamma\}$. Of course, $Z(t_0)$ is local extremum of the function $Z(t)$ located
at $t=t_0$.
\end{remark}

\subsection{}

In this paper we use the Riemann hypothesis together with some synthesis of properties of the sequences
\bdis
\{ t_0\},\ \{h_\nu(\tau)\},
\edis
where the numbers $h_\nu(\tau)$ are defined by the equation (comp. (1.2))
\be \label{1.4}
\begin{split}
& \vt_1[h_\nu(\tau)]=\pi\nu+\tau+\frac{\pi}{2},\ \nu=1,2,\dots,\ \tau\in [-\pi,\pi], \\
& \vt_1(t)=\frac t2\ln\frac{t}{2\pi}-\frac t2-\frac{\pi}{8}, \\
& \vt(t)=\vt_1(t)+\mcal{O}\left(\frac 1t\right),
\end{split}
\ee
in order to obtain the following theorem.

\begin{mydef1}
On the Riemann hypothesis we have the asymptotic formula
\be \label{1.5}
\begin{split}
& \int_T^{T+H} \sqrt{1+\{ Z'(t)\}^2}{\rm d}t=2\sum_{T\leq t_0\leq T+H}|Z(t_0)|+\\
& + \Theta H+\mcal{O}\left( T^{\frac{\Delta}{\ln\ln T}}\right), \\
& \Theta=\Theta(T,H)\in (0,1),\ H=T^{\epsilon},\ T\to\infty
\end{split}
\ee
for every fixed $\epsilon>0$.
\end{mydef1}

\begin{remark}
Geometric meaning of our asymptotic formula (1.5) is as follows: the arc length of the Riemann curve
\bdis
y=Z(t),\ t\in [T,T+H]
\edis
is asymptotically equal to the double of the sum of local maxima of the function
\bdis
|Z(t)|,\ t\in[T,T+H].
\edis
\end{remark}

\section{Discrete formulae -- Lemma 1}

\subsection{}

In this part of the paper we use the following formula
\be \label{2.1}
\begin{split}
& Z'(t)=-2\sum_{n<P}\frac{1}{\sqrt{n}}\ln\frac Pn\sin\{\vt-t\ln n\}+ \\
& + \mcal{O}(T^{-1/4}\ln T),\ P=\sqrt{\frac{T}{2\pi}},
\end{split}
\ee
that we have obtained in our work \cite{6}, (see (2.1)). Next, we obtain from (2.1) in the case
\bdis
\vt\to \vt_1
\edis
(see (1.4)) that
\be \label{2.2}
\begin{split}
& Z'(t)=-2\sum_{n<P}\frac{1}{\sqrt{n}}\ln\frac Pn\sin\{\vt_1-t\ln n\}+ \\
& + \mcal{O}(T^{-1/4}\ln T),\ H\in (0,\sqrt[4]{T}].
\end{split}
\ee
Let $S(a,b)$ denotes elementary trigonometric sum
\bdis
S(a,b)=\sum_{a\leq n\leq b}n^{it},\quad
1\leq a<b\leq 2a,\ b\leq \sqrt{\frac{t}{2\pi}}.
\edis
Then we obtain from (2.2) in the case of the sequence $\hh$ (see (1.4)) the following

\begin{lemma}
If
\be \label{2.3}
|S(a,b)|\leq A(\Delta)\sqrt{a},\ \Delta\in (0,1/6]
\ee
then ($h_\nu=h_\nu(0)$)
\be \label{2.4}
\begin{split}
& \sum_{T\leq h_{2\nu}\leq T+H}Z'[h_{2\nu}(\tau)]=-\frac 1\pi H\ln^2P\cos\tau+\mcal{O}(T^\Delta\ln^2T), \\
& \sum_{T\leq h_{2\nu+1}\leq T+H}Z'[h_{2\nu+1}(\tau)]=\frac 1\pi H\ln^2P\cos\tau+\mcal{O}(T^\Delta\ln^2T),
\end{split}
\ee
where $\mcal{O}$-estimates are uniform for $\tau\in [-\pi,\pi]$.
\end{lemma}

\begin{proof}
We obtain from (2.2) by (1.4)
\be \label{2.5}
\begin{split}
& Z'[\hh]=2(-1)^{\nu+1}\ln P\cos\tau-\\
& -2\sum_{2\leq n\leq P}\frac{1}{\sqrt{n}}\ln\frac Pn\cos\{\pi\nu-\hh\ln n+\tau\}+\\
& +\mcal{O}(T^{-1/4}\ln T),\ \hh\in [T,T+H].
\end{split}
\ee
\end{proof}

\subsection{}

Since (see \cite{5}, (23))
\be \label{2.6}
\sum_{T\leq h_\nu\leq T+H}1=\frac{1}{2\pi}H\ln\frac{T}{2\pi}+\mcal{O}(1)=\frac 1\pi H\ln P+\mcal{O}(1),
\ee
then we obtain from (2.5) (comp. \cite{4}, (59)-(61), \cite{6}, (51)-(53)) that
\be \label{2.7}
\sum_{T\leq h_\nu\leq T+H} Z'[\hh]=-2\bar{w}(T,H;\tau)+\mcal{O}(\ln^2T),
\ee
where
\bdis
\begin{split}
& \bar{w}=\frac 12(-1)^{\bar{\nu}}\sum_n \frac{1}{\sqrt{n}}\ln\frac Pn\cos\vp+ \\
& +\frac 12(-1)^{N+\bar{\nu}}\sum_n \frac{1}{\sqrt{n}}\ln\frac Pn\cos(\omega N+\vp)+ \\
& +\frac 12(-1)^{\bar{\nu}}\sum_n \frac{1}{\sqrt{n}}\ln\frac Pn\tan\frac \omega 2\sin\vp+ \\
& +\frac 12(-1)^{N+\bar{\nu}+1}\sum_n \frac{1}{\sqrt{n}}\ln\frac Pn\tan\frac \omega 2\sin(\omega N+\vp),
\end{split}
\edis
where
\bdis
\omega=\pi\frac{\ln n}{\ln P},\ \vp=h_{\bar{\nu}}(\tau)\ln n-\tau,\ n\in [2,P),
\edis
and
\bdis
\bar{\nu}=\min\{\nu:\ h_\nu\in [T,T+H]\},\ \bar{\nu}+N=\max\{\nu:\ h_\nu\in [T,T+H]\}.
\edis
Of course, we have
\bdis
\sum_{T\leq \hh\leq T+H}1=\sum_{T\leq h_\nu\leq T+H}1+\mcal{O}(1)
\edis
for any fixed $\tau\in [-\pi,\pi]$. Now, it is clear that the method \cite{6}, (54)-(64) implies by (2.3) that
\bdis
\bar{w}=\mcal{O}(T^\Delta \ln^2T)
\edis
uniformly for $\tau\in [-\pi,\pi]$, and consequently we obtain (see (2.7)) the estimate
\be \label{2.8}
\sum_{T\leq h_\nu\leq T+H} Z'[\hh]=\mcal{O}(T^\Delta\ln^2T)
\ee
uniformly for $\tau\in [-\pi,\pi]$.

\subsection{}

Next, we have (see (2.5), (2.6))
\bdis
\begin{split}
& \sum_{T\leq h_\nu\leq T+H} (-1)^\nu Z'[\hh]=-\frac{2}{\pi} H\ln^2P\cos\tau-2R+\mcal{O}(\ln^2P), \\
& R=\sum_{2\leq n<P}\frac{1}{\sqrt{n}}\ln\frac Pn\sum_{T\leq h_\nu\leq T+H}\cos\{\hh\ln n-\tau\}.
\end{split}
\edis
Since by (2.3) and \cite{6}, (65)-(79) we have the estimate
\bdis
R=\mcal{O}(T^\Delta\ln^2T)
\edis
then we obtain the formula
\be \label{2.9}
\begin{split}
& \sum_{T\leq h_\nu\leq T+H} (-1)^\nu Z'[\hh]=\\
& = -\frac 2\pi H\ln^2P\cos\tau+\mcal{O}(T^\Delta\ln^2T)
\end{split}
\ee
uniformly for $\tau\in [-\pi,\pi]$. \\

Finally, from (2.8), (2.9) formulae (2.4) follow.

\section{Integrals over disconnected sets -- Lemma 2}

Let (comp. \cite{7}, (3))
\be \label{3.1}
\begin{split}
& \mbb{G}_{2\nu}(x)=\{ t:\ h_{2\nu}(-x)<t<h_{2\nu}(x),\ t\in [T,T+H]\},\ x\in (0,\pi/2], \\
& \mbb{G}_{2\nu+1}(y)=\{ t:\ h_{2\nu+1}(-y)<t<h_{2\nu+1}(y),\ t\in [T,T+H]\},\ y\in (0,\pi/2], \\
& \mbb{G}_1(x)=\bigcup_{T\leq h_{2\nu}\leq T+H}\mbb{G}_{2\nu}(x), \\
& \mbb{G}_2(y)=\bigcup_{T\leq h_{2\nu+1}\leq T+H}\mbb{G}_{2\nu+1}(y).
\end{split}
\ee
The following lemma holds true.

\begin{lemma}
(2.3) implies
\be \label{3.2}
\begin{split}
& \int_{\mbb{G}_1(x)}Z'(t){\rm d}t=-\frac{2}{\pi}H\ln P\sin x+\mcal{O}(xT^\Delta \ln T), \\
& \int_{\mbb{G}_2(y)}Z'(t){\rm d}t=\frac{2}{\pi}H\ln P\sin y+\mcal{O}(yT^\Delta \ln T).
\end{split}
\ee
\end{lemma}

\begin{proof}
First of all we have (see (1.4), comp. \cite{7}, (51))
\bdis
\left(\frac{{\rm d}h_{2\nu}(\tau)}{{\rm d}\tau}\right)^{-1}=\vt_1'[\hhp]=\ln P+\mcal{O}\left(\frac HT\right).
\edis
Next, from (2.2) by (2.3) we obtain the estimate
\bdis
Z'(t)=\mcal{O}(T^\Delta\ln^2T),\ t\in[T,T+H]
\edis
(Abel transformation). Then we have (comp. \cite{7}, (52)) that
\be \label{3.3}
\begin{split}
& \int_{-x}^xZ'[\hhp]{\rm d}\tau=\int_{-x}^xZ'[\hhp]\left(\frac{{\rm d}h_{2\nu}(\tau)}{{\rm d}\tau}\right)^{-1}
\frac{{\rm d}h_{2\nu}(\tau)}{{\rm d}\tau}{\rm d}\tau= \\
& = \ln P\int_{h_{2\nu}(-x)}^{h_{2\nu}(x)}Z'(t){\rm d}t+\mcal{O}\left( x\frac HT T^\Delta\ln^2T\frac{1}{\ln T}\right)= \\
& = \ln P \int_{\mbb{G}_{2\nu}(x)}Z'(t){\rm d}t+\mcal{O}(xHT^{-5/6}\ln T).
\end{split}
\ee
Consequently, we obtain from the first formula in (2.4) by (2.6), (3.1), (3.3) the following asymptotic equality
\bdis
\begin{split}
& \int_{\mbb{G}_1(x)}Z'(t){\rm d}t=-\frac{2}{\pi}H\ln P\sin x+\\
& + \mcal{O}(xT^\Delta\ln T)+\mcal{O}(xH^2T^{-5/6}\ln^2T),
\end{split}
\edis
i.e. the first integral in (3.2). The second integral can be derived by a similar way.
\end{proof}

\section{An estimate from below -- Lemma 3}

The following lemma holds true.

\begin{lemma}
From (2.3) the estimate
\be \label{4.1}
\int_T^{T+H} |Z'(t)|{\rm d}t>\frac{4}{\pi}(1-\epsilon)H\ln P,\ P=\sqrt{\frac{T}{2\pi}},\ H\in [T^{\Delta+\epsilon},\sqrt[4]{T}]
\ee
follows, where $\epsilon>0$ is an arbitrarily small number.
\end{lemma}

\begin{proof}
Let (comp. \cite{8}, (10))
\bdis
\begin{split}
& \mbb{G}_1^+(x)=\{ t:\ Z'(t)>0,\ t\in \mbb{G}_1(x)\}, \\
& \mbb{G}_1^-(x)=\{ t:\ Z'(t)<0,\ t\in \mbb{G}_1(x)\}, \\
& \mbb{G}_1^0(x)=\{ t:\ Z'(t)=0,\ t\in \mbb{G}_1(x)\},
\end{split}
\edis
and the symbols
\bdis
\mbb{G}_2^+(y), \mbb{G}_2^-(y), \mbb{G}_2^0(y)
\edis
have similar meaning. Of course
\bdis
m\{\mbb{G}_1^0(x)\}=m\{\mbb{G}_2^0(y)\}=0.
\edis
Since the expressions (3.2) in the case
\bdis
H\in [T^{\Delta+\epsilon},\sqrt[4]{T}],\quad x,y\in (0,\pi/2]
\edis
are asymptotic formulae then from them we obtain the following inequalities
\be \label{4.2}
\begin{split}
& \frac{2}{\pi}(1-\epsilon)H\ln P<-\int_{\mbb{G}_1(\pi/2)}Z'(t){\rm d}t\leq \\
& \leq -\int_{\mbb{G}_1^-(\pi/2)}Z'(t){\rm d}t=\int_{\mbb{G}_1^-(\pi/2)}|Z'(t)|{\rm d}t, \\
&\frac{2}{\pi}(1-\epsilon)H\ln P<\int_{\mbb{G}_2(\pi/2)}Z'(t){\rm d}t\leq
\int_{\mbb{G}_2^+(\pi/2)}|Z'(t)|{\rm d}t.
\end{split}
\ee
Since
\bdis
\mbb{G}_1^-(\pi/2)\cup \mbb{G}_2^+(\pi/2)\subset [T,T+H],\ \mbb{G}_1^-(\pi/2)\cap \mbb{G}_2^+(\pi/2)=\emptyset
\edis
then by (4.2) needful estimate
\bdis
\begin{split}
& \int_T^{T+H}|Z'(t)|{\rm d}t\geq \int_{\mbb{G}_1^-(\pi/2)}|Z'(t)|{\rm d}t+\int_{\mbb{G}_2^+(\pi/2)}|Z'(t)|{\rm d}t> \\
& > \frac{4}{\pi}(1-\epsilon)H\ln P.
\end{split}
\edis
follows.
\end{proof}

\section{Quadrature formula -- Lemma 4}

The following lemma holds true.

\begin{lemma}
On Riemann hypothesis we have the following asymptotic formula
\be \label{5.1}
\begin{split}
& \int_T^{T+H}|Z'(t)|{\rm d}t=2\sum_{T\leq t_0\leq T+H}|Z(t_0)|+ \\
& + \mcal{O}\left( T^{\frac{A}{\ln\ln T}}\right),\quad H\in [T^\mu,\sqrt[4]{T}],
\end{split}
\ee
where $0<\mu$ is an arbitrary small number.
\end{lemma}

\begin{proof}
First of all, we have on Riemann hypothesis the following two Littlewood's estimates
\be \label{5.2}
\gamma''-\gamma'<\frac{A}{\ln\ln\gamma'},\ \gamma'\to\infty
\ee
(see \cite{2}, p. 237), and
\be \label{5.3}
Z(t)=\mcal{O}\left( t^{\frac{A}{\ln\ln t}}\right),\ t\to\infty
\ee
(see \cite{13}, p. 300). Next, on Riemann hypothesis we have the following basic configuration (see Remark 2)
\be \label{5.4}
\gamma'<t_0<\gamma'';\ t_0\in [T,T+H].
\ee
Now, there are following possibilities (see (5.4)): either
\be \label{5.5}
\begin{split}
& Z(t)>0,\ t\in (\gamma',\gamma'') \ \Rightarrow \\
& Z'(t)>0,\ t\in (\gamma',t_0),\ Z'(t)<0,\ t\in (t_0,\gamma''),
\end{split}
\ee
or
\be \label{5.6}
\begin{split}
& Z(t)<0,\ t\in(\gamma',\gamma'') \ \Rightarrow \\
& Z'(t)<0,\ t\in(\gamma',t_0),\ Z'(t)>0,\ t\in (t_0,\gamma'').
\end{split}
\ee
Consequently, (5.5) and (5.6) imply that
\be \label{5.7}
\int_{\gamma'}^{\gamma''}|Z'(t)|{\rm d}t=2|Z(t_0)|,\ \forall t_0\in [T,T+H].
\ee
Similarly, we obtain (see (5.2), (5.3)) the estimates
\be \label{5.8}
\int_{\bar{\gamma}'}^{\bar{\gamma}''}|Z'(t)|{\rm d}t,\ \int_{\bar{\bar{\gamma}}'}^{\bar{\bar{\gamma}}''}|Z'(t)|{\rm d}t=
\mcal{O}\left(\frac{T^{\frac{A}{\ln\ln T}}}{\ln\ln T}\right)
\ee
in the following cases
\bdis
\bar{\gamma}'<T\leq t_0<\bar{\gamma}'',\ \bar{\bar{\gamma}}'<t_0\leq T+H<\bar{\bar{\gamma}}''.
\edis
Now, our formula (5.1) follows from (5.7), (5.8).
\end{proof}

\section{Proof of Theorem}

We use the following formula
\be \label{6.1}
\begin{split}
& \int_T^{T+H}\sqrt{1+\{ Z'(t)\}^2}{\rm d}t= \\
& = \int_T^{T+H}|Z'(t)|{\rm d}t+\int_T^{T+H}\frac{1}{\sqrt{1+\{ Z'(t)\}^2}+|Z'(t)|}{\rm d}t.
\end{split}
\ee
Since
\bdis
0<\frac{1}{\sqrt{1+\{ Z'(t)\}^2}+|Z'(t)|}\leq 1
\edis
and
\be \label{6.2}
\left. \frac{1}{\sqrt{1+\{ Z'(t)\}^2}+|Z'(t)|}\right|_{t=t_0}=1,\ t_0\in [T,T+H] ,
\ee
i.e. the inequality (6.2) holds true for the finite set of values, then the mean-value theorem gives
\be \label{6.3}
\int_T^{T+H}\frac{1}{\sqrt{1+\{ Z'(t)\}^2}+|Z'(t)|}{\rm d}t=\Theta H,\
\Theta=\Theta(T,H)\in (0,1).
\ee
Next, we obtain by (4.1), (5.1), ($\mu\leq \epsilon$), the inequality
\be \label{6.4}
\begin{split}
& \frac{4}{\pi}(1-\epsilon)H\ln P<\int_T^{T+H}|Z'(t)|{\rm d}t=\\
& = 2\sum_{T\leq t_0\leq T+H}|Z'(t_0)|+\mcal{O}\left( T^{\frac{A}{\ln\ln T}}\right).
\end{split}
\ee
Hence, by (6.1)-(6.4) the formula (1.5) follows for
\be \label{6.5}
H\in [T^{\Delta+\epsilon},\sqrt[4]{T}].
\ee
Since the Riemann hypothesis implies Lindel\" of hypothesis a it implies that $\Delta=\epsilon$ (comp. \cite{1}, p. 89), then
we obtain from (6.5) that
\bdis
H=T^{2\epsilon};\ 2\epsilon\rightarrow \epsilon ,
\edis
(see (1.5)).

\appendix

\section{Influence of Jacob's ladders}

If
\bdis
\vp_1\{[\mT,\widering{T+H}]\}=[T,T+H],
\edis
then from (1.5) we obtain (see \cite{10}, (9.7)) the formula
\be \label{A.1}
\begin{split}
& \int_{\mT}^{\widering{T+H}}\sqrt{1+\{ Z'_{\vp_1}[\vp_1(t)]\}^2}\left|\zf\right|^2{\rm d}t\sim \\
& \sim \left\{ 2\sum_{T\leq t_0\leq T+H}|Z(t_0)|+\Theta H+\mcal{O}\left( T^{\frac{A}{\ln\ln T}}\right)\right\}\ln T,\ T\to\infty.
\end{split}
\ee
From (A.1) we obtain by mean-value theorem that
\be \label{A.2}
\begin{split}
&\int_{\mT}^{\widering{T+H}}\sqrt{1+\{ Z'_{\vp_1}[\vp_1(t)]\}^2}{\rm d}t\sim \\
& \sim \frac{\ln T}{\left|\zeta\left( \frac 12+i\alpha\right)\right|^2}
\left\{ 2\sum_{T\leq t_0\leq T+H} \left|\zeta\left( \frac 12+it_0\right)\right|+\Theta H+\mcal{O}\left( T^{\frac{A}{\ln\ln T}}\right)\right\}, \\
& \alpha\in (\mT,\widering{T+H}).
\end{split}
\ee

\begin{remark}
Since we have (see \cite{10}, (8.5))
\bdis
\rho\{ [T,T+H];[\mT,\widering{T+H}]\}\sim (1-c)\pi(T)>(1-\epsilon)(1-c)\frac{T}{\ln T},\ T\to\infty,
\edis
where $\rho$ denotes the distance of corresponding segments and $\pi(T)$ is the prime-counting function and $c$ is the Euler constant,
then the formula (A.2) gives strongly non-local expression for the integral on the left-hand side of (A.2).
\end{remark}

\thanks{I would like to thank Michal Demetrian for helping me with the electronic version of this work.}

\end{document}